\documentclass{cmslatex}
\usepackage[utf8]{inputenc}
\usepackage[paperwidth=7in, paperheight=10in, margin=.875in]{geometry}
\usepackage[backref,colorlinks,linkcolor=red,anchorcolor=green,citecolor=blue]{hyperref}
\usepackage{booktabs,multirow}
\usepackage[noend]{algpseudocode}
\usepackage[pagewise]{lineno}
\usepackage{algorithmicx,algorithm,setspace}
\usepackage{amsfonts,amssymb}
\usepackage{amsmath}
\usepackage{graphicx,subfigure}
\usepackage{cite}
\usepackage{float}
\usepackage{todonotes}
\usepackage{enumerate}
\usepackage{ulem,marvosym}
\usepackage{xcolor}
\usepackage{array}
\usepackage{colortbl}
\usepackage{nicematrix}
\usepackage{tikz}
\usepackage{enumitem}
\usepackage{mathtools}
\usepackage{soul} 
\usepackage{enumitem}

\soulregister\eqref7
\soulregister\cite7
\definecolor{green}{rgb}{0.0,0.55,0.13}

\renewcommand{\theequation}{\arabic{section}.\arabic{equation}}

\newcommand{\abs}[1]{\left|#1\right|}
\newcommand{\brac}[1]{\left(#1\right)}
\newcommand{\norm}[1]{\left\Vert#1\right\Vert}

\newcommand{\ba}{\boldsymbol{a}}
\newcommand{\bb}{\boldsymbol{b}}
\newcommand{\bx}{\boldsymbol{x}}
\newcommand{\by}{\boldsymbol{y}}

\newcommand{\bp}{\boldsymbol{p}}
\newcommand{\bq}{\boldsymbol{q}}
\newcommand{\br}{\boldsymbol{r}}
\newcommand{\bs}{\boldsymbol{s}}

\newcommand{\mbr}{\mathbb{R}}
\newcommand{\mbn}{\mathbb{N}}
\newcommand{\bu}{\boldsymbol{u}}
\newcommand{\bv}{\boldsymbol{v}}

\newcommand{\bphi}{\boldsymbol{\phi}}
\newcommand{\bpsi}{\boldsymbol{\psi}}
\newcommand{\bhe}{\boldsymbol{\eta}}
\newcommand{\bxi}{\boldsymbol{\xi}}
\newcommand{\veps}{\varepsilon}

\allowdisplaybreaks
\begin{document}
\begin{sloppypar}
\title{FINOM: Fast Sinkhorn on Non-uniform Meshes}
\author{Qihao Cheng\footnote{Department of Mathematical Sciences, Tsinghua University, Beijing 100084, China 
(cqh22@mails.tsinghua.edu.cn).}, Qichen Liao\footnote{Department of Mathematical Sciences, Tsinghua University, Beijing 100084, China 
(lqc20@tsinghua.org.cn).}, Hao Wu\footnote{Corresponding author. Department of Mathematical Sciences, Tsinghua University, Beijing 100084, China 
(hwu@tsinghua.edu.cn).}, and Shuai Yang\footnote{Department of Mathematical Sciences, Tsinghua University, Beijing 100084, China 
(s-yang21@mails.tsinghua.edu.cn).}}
\maketitle
\begin{abstract}
   A linear-complexity algorithm for computing the Wasserstein-1 distance on non-uniform meshes is proposed.   This work extends the fast Sinkhorn algorithms from [Q. Liao et al., Commun. Math. Sci., 20(2022)] and [Q. Liao et al., J. Sci. Comput., 98 (2024)] to non-uniform meshes. In those prior works, a distinctive collinear structure of the kernel matrix on uniform meshes was identified, enabling \(O(N)\) acceleration via dynamic programming. While non-uniform meshes are prevalent in practical applications like computational fluid dynamics and finance, their lack of collinearity has hindered direct acceleration. In this paper, we introduce the concept of a ``dividing index'', which partitions the kernel matrix into two blocks. We demonstrate that each block exhibits a quasi-collinear property, a generalization of the structure found in uniform meshes. Leveraging this insight, we develop \textbf{F}ast S\textbf{I}nkhorn algorithm on \textbf{NO}n-uniform \textbf{M}eshes (\textbf{FINOM}), a dynamic programming approach that reduces the per-iteration complexity of the Sinkhorn algorithm from \(O(N^2)\) to \(O(N)\). Extensive numerical experiments on 1D and 2D problems confirm these improvements, achieving speed-ups of several orders of magnitude while maintaining accuracy.
\end{abstract}
\section{Introduction}
\label{sec:intro}
The theory of Optimal Transport (OT), which seeks the most cost-effective way to transform one probability distribution into another, provides a powerful framework for comparing measures via the Wasserstein metric~\cite{peyre2019computational,villani2009optimal}. As the Wasserstein metric provides a global comparison between probability distributions, it offers significant advantages over traditional metrics and has been widely applied in numerous real-world applications, such as image processing~\cite{zhan2021unbalanced,zhu2007image,museyko2009application}, machine learning~\cite{arjovsky2017wasserstein,kolouri2017optimal}, inverse problems~\cite{chen2018quadratic,engquist2016optimal,engquist2022optimal,zhou2023wasserstein}, communication engineering~\cite{wu2023communication,ye2022optimal}.

In this work, we focus on computing the Wasserstein-1 distance between two discrete measures $\mu=\sum_{i=1}^{N}u_i\delta_{\bx_i}$ and $\nu=\sum_{j=1}^{M}v_j\delta_{\by_j}$\footnote{The specific details regarding the positions $x_i$ and $y_j$ will be provided in subsequent sections.} on potentially non-uniform meshes, defined as the solution to the linear program:
\begin{equation}\label{eq:w1Definition}
    \begin{aligned}
        &\min_{\Gamma\in\Pi(\bu,\bv)}\left\langle C,\;\Gamma\right\rangle, \\
        &\Pi(\bu,\bv) \triangleq\left\{\Gamma\in\mbr^{N\times M}_+\;\big\vert \; \Gamma\mathbf{1}_M=\bu,\;\Gamma^T\mathbf{1}_N=\bv\right\}. \\
    \end{aligned}
\end{equation}
Here, $\Gamma=\left[\gamma_{ij}\right]_{N\times M}$ is the transport plan and
$C=\left[c_{ij}\right]=[\left\Vert\bx_i-\by_j\right\Vert_1]\in\mbr^{N\times M}_+$ is the cost matrix.\footnote{Our method and analysis apply to the general case of $N\neq M$. The assumption $N=M$ is made hereafter for notational simplicity and without loss of generality. } 

Recently, a variety of numerical algorithms have been developed to tackle the optimal transport problem from diverse perspectives, including linear programming methods~\cite{oberman2015efficient,yang2021fast}, primal-dual algorithms~\cite{hu2022efficient}, methods based on solving the Monge-Amp\`{e}re equation~\cite{de2014monge,benamou2000computational,huesmann2019benamou}, proximal splitting~\cite{metivier2016optimal}, and proximal block coordinate descent methods~\cite{hu2023global}. 

A widely adopted approach with provable convergence is the Sinkhorn algorithm~\cite{sinkhorn1967diagonal,altschuler2017near,lin2019efficient}, which employs entropy regularization to approximate solutions to optimal transport and its variants~\cite{cuturi2013sinkhorn,wu2022double,scetbon2021low,le2021robust}.
Each iteration of the Sinkhorn algorithm alternately updates two scaling vectors via a matrix-vector multiplication and an element-wise division, yielding a per-iteration computational complexity of $O(N^2)$. The Sinkhorn iteration procedure can be accelerated by leveraging massively parallel computing architectures~\cite{tithi2020efficient,feydy2019fast}. Separately, a number of approximate methods have been proposed to reduce this computational burden, including heat kernel approximation and convolution techniques~\cite{solomon2015convolutional}, low-rank approximations~\cite{altschuler2019massively,gasteiger2021scalable,scetbon2021low}, and importance sparsification~\cite{li2023importance}.

For computing the Wasserstein-1 metric on uniform meshes, recently, a unified fast Sinkhorn approach, introduced in the complementary works~\cite{liao2022fast,liao2024fast}, leverages the special structure of the kernel matrix to accelerate each iteration without sacrificing accuracy. In each iteration, dynamic programming significantly reduces the complexity of the matrix-vector multiplication from $O\brac{N^{2}}$ to $O\brac{N}$. 

In practical applications, non-uniform meshes are frequently employed. They are particularly suitable for problems involving complex spatial geometries or concentrated probability densities, with prominent applications in computational fluid dynamics~\cite{devloo1988hp, mavriplis1994adaptive}, mathematical finance~\cite{wang2019variable}, among others. For OT in Euclidean space, when either the source or target distribution is relatively concentrated, using a non-uniform mesh can significantly reduce the required number of grid points compared to a uniform mesh, thereby lowering computational costs. However, it is challenging to reduce the computational cost of the Sinkhorn algorithm on non-uniform meshes to $O\brac{N}$ because the collinear property of the kernel matrix is no longer applicable, preventing the direct application of the Fast Sinkhorn method from~\cite{liao2022fast,liao2024fast}. Moreover, the \(O(N^{2})\) complexity of the standard Sinkhorn algorithm would offset the computational savings achieved through non-uniform meshing.

In this paper, we propose \textbf{F}ast s\textbf{I}nkhorn algorithm on \textbf{NO}n-uniform \textbf{M}eshes (\textbf{FINOM}). To address the loss of the collinear property in this general setting, we introduce a key innovation: a matrix partitioning strategy based on a novel ``dividing index''. This index decomposes the kernel matrix into two sub-blocks, each of which is shown to exhibit a quasi-collinear property. By leveraging this structure, we develop an adapted dynamic programming scheme that performs the core matrix-vector multiplication with linear complexity. Numerical experiments demonstrate that our algorithm achieves dramatic speed-ups over the standard Sinkhorn method on non-uniform meshes, while preserving the exactness of the original matrix-vector product computations.  

The remainder of this manuscript is arranged as follows. Section~\ref{sec:1d} details our core contribution: the fast Sinkhorn algorithm for 1D non-uniform meshes. It introduces the dividing index, establishes the quasi-collinear structure of the partitioned kernel matrix, and presents the resulting linear-complexity dynamic programming for matrix-vector multiplication. The generalization of our method to higher-dimensional spaces is presented in Section~\ref{sec:2d}. Section~\ref{sec:experiments} is devoted to numerical experiments on 1D and 2D problems, validating the linear-complexity and superior performance of our proposed algorithm. We conclude the paper with a summary and outlook on future work in Section~\ref{sec:conclusion}.

\section{One-dimensional FINOM}
\label{sec:1d}
In this section, we discuss the fast algorithm for the Wasserstein-1 metric on 1D non-uniform meshes. Without loss of generality, we assume that the locations $\bx=\brac{x_1,x_2,\cdots,x_N}$ and $\by=\brac{y_1,y_2,\cdots,y_N}$ are sorted in ascending order, i.e., $x_1< x_2<\cdots< x_N$ and $y_1< y_2<\cdots< y_N$. 

\subsection{Sinkhorn algorithm and the kernel matrix}
The Sinkhorn algorithm\cite{cuturi2013sinkhorn, sinkhorn1967diagonal} addresses an entropy-regularized version of Problem~\eqref{eq:w1Definition}:
\begin{equation}\label{eq:regularized problem}
\min_{\Gamma\in\Pi(\bu,\bv)}\sum_{i,j}C_{ij}\gamma_{ij} + \varepsilon \sum_{i,j}\gamma_{ij}\brac{\log \gamma_{ij}-1}, 
\end{equation}
yielding an efficient and approximate optimal transport to the primal problem. The solution of Problem~\eqref{eq:regularized problem} is unique and has the form
\begin{equation*}
\gamma_{ij}=\phi_ik_{ij}\psi_j, \quad \text{where}\quad k_{ij}=e^{-\frac{c_{ij}}{\varepsilon}}, 
\end{equation*} 
or equivalently, $\Gamma=\operatorname{diag}\left(\bphi\right)K\operatorname{diag}\left(\bpsi\right)$. Here $K=\left[k_{ij}\right]= e^{-C/\varepsilon}\in \mbr^{N\times N}$
is the kernel matrix, and the scaling vectors $\bphi=\left[\phi_i\right] \in \mbr^N$ and $\bpsi=\left[\psi_j\right] \in \mbr^N$
can be obtained by the iterative method:
\begin{equation}
    \label{eq:iteration}
    \bpsi^{(l+1)}=\bv \oslash \brac{K^{\top}\bphi^{(l)}}, \quad \bphi^{(l+1)}=\bu \oslash \brac{K\bpsi^{(l+1)}}.
\end{equation}
Here $\oslash$ represents point-wise division and $l$ denotes the iterative counter. 

Matrix-vector multiplication (e.g., $K\bpsi$) typically requires $O(N^2)$ computations and dominates the complexity of a Sinkhorn iteration in~\eqref{eq:iteration}. In previous works~\cite{liao2022fast,liao2024fast}, by leveraging the special structure of the kernel matrix on a uniform mesh, the complexity of such matrix-vector multiplication has been effectively reduced to $O(N)$, thus accelerating the Sinkhorn algorithm. To extend this linear-complexity result to non-uniform meshes, we introduce a block partitioning framework, which is enabled by a key tool called the dividing index.

\subsection{Block partitioning via the dividing index}
In \cite{liao2022fast,liao2024fast}, due to the uniformity of the mesh, the lower and upper triangular parts of the kernel matrix exhibit a collinear structure, which are defined as the Lower-ColLinear Triangular Matrix and Upper-ColLinear Triangular Matrix in \cite[Section 3.1]{liao2024fast}. However, this structure does not hold for the kernel matrix under non-uniform mesh grids. Instead, an analogous structure is sought by partitioning the kernel matrix. This partitioning is formalized using the dividing index:
\begin{definition}[Dividing index]
Let $\bx=\brac{x_1,x_2,\cdots,x_N}$ and $\by=\brac{y_1,y_2,\cdots,y_N}$ be sorted meshes with $x_1<x_2<\dots<x_N$ and $y_1<y_2<\dots<y_N$. Define $\zeta:\{1,2,\cdots,N\}\to \{0,1,2,\cdots,N\}$ as a mapping, where for each $x_i$ (the coordinate corresponding to row $i$ of the kernel matrix $K$):
\begin{itemize}
    \item If $y_1\leq x_i < y_N$, then $\zeta(i)$ is the unique index satisfying $y_{\zeta(i)}\leq x_i < y_{\zeta(i)+1}$.
    \item If $x_i<y_1$, then $\zeta(i)=0$.
    \item If $x_i\geq y_N$, then $\zeta(i)=N$.
\end{itemize}
\end{definition}

With this dividing index, we can now impose a block partition on the kernel matrix. The dividing index decomposes $K$ into two complementary sub-blocks: the lower block $K_L=[\check{k}_{ij}]\in \mbr^{N\times N}$, where   
\begin{equation*}
    \check{k}_{ij}=
    \begin{cases}
        k_{ij}=e^{-(x_i-y_j)/\veps},\ &1\leq i\leq N,\; 1\leq j \leq \zeta(i),   \\
        0, \  &1\leq i\leq N, \;\zeta(i)< j \leq N,  
    \end{cases}
    \; 
\end{equation*}
and the upper block $K_U=[\hat{k}_{ij}]\in \mbr^{N\times N}$, where   
\begin{equation*}
    \hat{k}_{ij}=
    \begin{cases}
        0,\ &1\leq i\leq N, \; 1\leq j \leq \zeta(i),   \\
        k_{ij} =e^{-(y_j-x_i)/\veps}, \ &1\leq i\leq N, \; \zeta(i)< j \leq N. 
    \end{cases}
    \; 
\end{equation*}
Through this decomposition, we observe the following relation: For $1\leq i_1\leq i_2 \leq N,$

\begin{equation}
    \label{eq:collinear}
		\begin{aligned}
			&\check{k}_{i_1j}\;/\;\check{k}_{i_2j}\;=\;e^{-(x_{i_1}-x_{i_2})/\veps},\quad 1\leq j \leq \zeta(i_1),\\
			&\hat{k}_{i_1j}\;/\;\hat{k}_{i_2j}\;=\;e^{-(x_{i_2}-x_{i_1})/\veps},\quad \zeta(i_2)<j\leq N.
		\end{aligned}
\end{equation}
This relation shows that both $K_L$ and $K_U$ exhibit a row-wise constant-ratio property wherever the corresponding entries are nonzero.

Crucially, the dividing index also induces a staircase nonzero pattern. Consider $K_L$ as an example, the nonzero entries in row $i$ form a contiguous block from column 1 to column $\zeta(i)$, and these blocks are non-decreasing in length (i.e., $\zeta(i)\leq\zeta(i+1)$), forming a lower staircase. Such specific property, namely staircase nonzero support, together with constant ratios between consecutive rows in overlapping nonzero parts, enables a compact $O(N)$-space representation of $K_L$, as established by the following theorem:



\begin{theorem}[Vector representation of $K_L$]\label{thm:representation of k_L}
The sub-block $K_L\in\mathbb{R}^{N\times N}$ induced by a dividing index $\zeta$, admits the following representation with a ratio vector $\br=[r_i]_{1\leq i\leq N-1}\in \mbr^{N-1}$, and a set of $N$ block-edge vectors $\mathcal{H}=\left\{\bhe_1,
\bhe_2,\dots,\bhe_N\right\}$ with a total of $\zeta(N)$ entries:
\setlength{\arraycolsep}{8pt}
\renewcommand{\arraystretch}{1.5}
\begin{equation}\label{eq:matrix-rep of K_L}
K_L = 
\begin{pmatrix}
\bhe_1 \\
r_1 \bhe_1 &\;\; \bhe_2 \\
r_1 r_2 \bhe_1 &\;\; r_2 \bhe_2 &\;\; \bhe_3 \\
\vdots &\;\; \vdots &\;\; \vdots &\quad \ddots \\
\displaystyle\prod_{t=1}^{N-1} r_t \bhe_1 & \;\;\;\displaystyle\prod_{t=2}^{N-1} r_t \bhe_2 &\;\; \cdots &\quad \cdots &\quad\bhe_N  & \quad&
\end{pmatrix},
\end{equation}
where 
\begin{align}
r_i &= \frac{\check{k}_{i+1,1}}{\check{k}_{i1}}, \quad i=1,\dots,N-1, \label{eq:ratio}\\
\bhe_i &= [\check{k}_{i,\zeta(i-1)+1}\;,\dots,\;\check{k}_{i,\zeta(i)}]\in \mbr^{\zeta(i)-\zeta(i-1)}, \quad i=1,\dots,N.\footnotemark \label{eq:eta}
\end{align}
\end{theorem}
\footnotetext{For notational convenience, we additionally define $\zeta(0)=0$. Define $r_i=1$ if the denominator $\check{k}_{i1}$ is zero. The dimension of $\bhe_i$ vanishes when $\zeta(i)=\zeta(i-1)$.} 
\begin{proof}
Consider any entry $\check{k}_{ij}$ of $K_L$. If there is some $s<i$ such that $\zeta(s-1) < j \leq \zeta(s)$, then we have
\begin{equation*}
    \check{k}_{ij}=\check{k}_{sj}\;\prod\limits_{t=s}^{i-1}\frac{\check{k}_{t+1,j}}{\check{k}_{tj}}=\check{k}_{sj}\;\prod\limits_{t=s}^{i-1}\frac{\check{k}_{t+1,1}}{\check{k}_{t1}}=\check{k}_{sj}\;\prod\limits_{t=s}^{i-1} r_t,
\end{equation*}
where $\check{k}_{sj}\in \bhe_s$. Otherwise $j>\zeta(i-1)$, then $\check{k}_{ij}\in\bhe_i$ or $\check{k}_{ij} = 0$. 
\end{proof}

An analogous conclusion holds for $K_U$ as well:

\begin{theorem}[Vector representation of $K_U$]\label{thm:representation of K_U}
The sub-block $K_U\in\mathbb{R}^{N\times N}$ induced by a dividing index $\zeta$, admits the following representation with a ratio vector $\br'=[r'_i]_{1\leq i\leq N-1}\in \mbr^{N-1}$, and a set of $N$ block-edge vectors $\mathcal{H}'=\left\{\bhe'_1,
\bhe'_2,\dots,\bhe'_N\right\}$ with a total of $\left(N-\zeta(1)\right)$ entries:
\setlength{\arraycolsep}{8pt}
\renewcommand{\arraystretch}{1.5}
\begin{equation}\label{eq:matrix-rep of K_U}
K_U = 
\begin{pmatrix}
\quad&\bhe'_1 \;\;\;& \cdots\;\;& \cdots & \displaystyle\prod_{t=1}^{N-2} r'_t \bhe'_{N-1} &\;\;\displaystyle\prod_{t=1}^{N-1} r'_t \bhe'_N\\
& &\ddots &\vdots&\;\;\vdots & \vdots\\
&& & \;\;\bhe'_{N-2} &\;\;r'_{N-2}\bhe'_{N-1} &\;\;r'_{N-2}r'_{N-1}\bhe'_N \\
&& & &\;\; \bhe'_{N-1} & \;\;r'_{N-1}\bhe'_N \\
&& & & &\;\;\bhe'_N  
\end{pmatrix},
\end{equation}
where 
\begin{align}
    r'_{i} &= \frac{\hat{k}_{i,N}}{\hat{k}_{i+1,N}}, \quad i=1,\dots,N-1,\\
    \bhe'_i &= [\hat{k}_{i,\zeta(i)+1}\;,\dots,\;\hat{k}_{i,\zeta(i+1)}], \quad i=1,\dots,N.\footnotemark
\end{align}
\end{theorem}

\footnotetext{For notational convenience, we additionally define $\zeta(N+1)=N$. Define $r'_i=1$ if the denominator $\hat{k}_{i+1,N}$ is zero. The dimension of $\bhe'_i$ vanishes when $\zeta(i)=\zeta(i+1)$.}

To illustrate this vector representation, consider the following example:
\begin{example}
Consider two meshes $\bx=(1,3,7,9,12)$, $\by=(2,5,6,9,10)$. The kernel matrix $K=[k_{ij}]\in \mbr^{5\times5}$ is given by
\renewcommand{\arraystretch}{1.5}
\setlength{\arraycolsep}{8pt}
$$
K = \begin{pmatrix}
\lambda^1 & \lambda^4 & \lambda^5 & \lambda^8 & \lambda^9 \\
\lambda^1 & \lambda^2 & \lambda^3 & \lambda^6 & \lambda^7 \\
\lambda^5 & \lambda^2 & \lambda^1 & \lambda^2 & \lambda^3 \\
\lambda^7 & \lambda^4 & \lambda^3 & \lambda^0 & \lambda^1 \\
\lambda^{10} & \lambda^7 &\lambda^6 & \lambda^3 & \lambda^2
\end{pmatrix},
$$
where $\lambda=e^{-1/\varepsilon}$. This yields a dividing index $\zeta = (0, 1, 3, 4, 5)$. The corresponding sub-blocks are 
$$
K_L = \begin{pmatrix}
0 & 0 & 0 & 0 & 0 \\
\lambda^1 & 0 & 0 & 0 & 0 \\
\lambda^5 & \lambda^2 & \lambda^1 & 0 & 0 \\
\lambda^7 & \lambda^4 & \lambda^3 & \lambda^0 & 0 \\
\lambda^{10} & \lambda^7 &\lambda^6 & \lambda^3 & \lambda^2
\end{pmatrix},\quad
K_U = \begin{pmatrix}
\lambda^1 & \lambda^4 & \lambda^5 & \lambda^8 & \lambda^9 \\
0 & \lambda^2 & \lambda^3 & \lambda^6 & \lambda^7 \\
0 & 0 & 0 & \lambda^2 & \lambda^3 \\
0 & 0 & 0 & 0 & \lambda^1 \\
0 & 0 & 0 & 0 & 0
\end{pmatrix}.
$$

By Theorem~\ref{thm:representation of k_L}, the ratio vector $\br$ of $K_L$ is given by 
\begin{equation*}
    \br=[1,\lambda^{5-1},\lambda^{7-5},\lambda^{10-7}]=[1,\lambda^{4},\lambda^{2},\lambda^{3}],
\end{equation*}
and the block-edge vectors $\mathcal{H}$ are
$$
\bhe_1 = [], \quad
\bhe_2 =[\lambda^{1}], \quad
\bhe_3 = [\lambda^{2},\lambda^{1}], \quad
\bhe_4 = [\lambda^{0}], \quad
\bhe_5 = [\lambda^{2}].
$$
Starting from $\bhe_1=[]$, we can reconstruct the nonzero part of $K_L$ iteratively: for $i=1,\ldots,4$, scale the previously recovered row by $r_i$ to obtain the shared prefix of row $i+1$, and then append the block-edge vector $\bhe_{i+1}$.
Similarly, utilizing Theorem~\ref{thm:representation of K_U}, we derive the representation of $K_U$ as
\begin{equation*}
    \br' = [\lambda^{9-7},\lambda^{7-3},\lambda^{3-1},1] = [\lambda^{2},\lambda^{4},\lambda^{2},1],
\end{equation*}
and
$$
\bhe'_1 = [\lambda^{1}], \quad
\bhe'_2 =[\lambda^{2},\lambda^{3}], \quad
\bhe'_3 = [\lambda^{2}], \quad
\bhe'_4 = [\lambda^{1}], \quad
\bhe'_5 = [].
$$

\end{example}

We refer to the property that $K_L$ and $K_U$ admit a representation via a ratio vector and block-edge vectors as \textbf{quasi-collinearity}. This notion generalizes the collinear triangular matrix property introduced in~\cite{liao2024fast}. Such a property enables a linear-complexity dynamic programming approach for matrix-vector multiplication, as detailed in the next subsection.

\subsection{Fast matrix-vector multiplication}
\label{subsec:fmvm}
The key step of the Sinkhorn algorithm is to iteratively update $\bphi$ and $\bpsi$ through~\eqref{eq:iteration}, where the $O(N^2)$ matrix-vector multiplications dominate the complexity. In this subsection, we present an efficient method that computes $K\bpsi$ and $K^\top\bphi$ with linear complexity. The following discussion will focus on the computation of $K\bpsi$, as the approaches for the remaining cases are similar.

First, we decompose $K\bpsi=K_L\bpsi+K_U\bpsi \triangleq\bp+\bq$, where
\begin{equation}
\label{eq:decompose to p+q}
    p_i\;=\;\sum\limits_{j=1}^{\zeta(i)} \check{k}_{ij}\psi_j,\quad
    q_i=\sum\limits_{j=\zeta(i)+1}^N \hat{k}_{ij}\psi_j,\quad i=1,2,\cdots,N.
\end{equation}

By Theorem~\ref{thm:representation of k_L}, $K_L$ is fully specified by a ratio vector $\br$ and block-edge vector set $\mathcal{H}=\{\bhe_i\}_{i=1}^N$. We similarly partition $\bpsi_{1:\zeta(N)}$ into $N$ segments: $\bpsi_{1:\zeta(N)}=[\bpsi_1^\top,\bpsi_2^\top,\dots,\bpsi_N^\top]^\top$, where
\begin{equation*}
   \bpsi_i=[\psi_{\zeta(i-1)+1},\dots,\psi_{\zeta(i)}]^\top,\quad i=1,\dots,N.
\end{equation*}
Note that for each $i$, $\bpsi_i$ has the same dimension as $\bhe_i$.
Instead of computing $p_i$ directly, we start with the initial result $p_1=\sum\limits_{j=1}^{\zeta(1)}\check{k}_{1j}\psi_j=\bhe_1\cdot\bpsi_1$, and use the recursive computation given by
\begin{equation}
    \begin{aligned}
    \label{eq:p by dynamic programming}
    p_{i+1} = \sum_{s=1}^{i}\;\;\sum_{j=\zeta(s-1)+1}^{\zeta(s)}&\check{k}_{i+1,j}\psi_j+\sum_{j=\zeta(i)+1}^{\zeta(i+1)}\check{k}_{i+1,j}\psi_j\\
    &=\sum_{s=1}^i \Bigl(\prod_{t=s}^i r_t \bhe_s\Bigr)\cdot\bpsi_s + \bhe_{i+1}\cdot\bpsi_{i+1}\\
    &\qquad= r_i \Bigl( \sum_{s=1}^{i-1} \Bigl(\prod_{t=s}^{i-1} r_t \bhe_s \Bigr) \cdot\bpsi_s + \bhe_i\cdot\bpsi_i \Bigr) + \bhe_{i+1}\cdot\bpsi_{i+1} \\
    &\qquad\quad\qquad\qquad\qquad= r_i p_i + \bhe_{i+1}\cdot\bpsi_{i+1}, \quad i=1,\dots,N-1.
    \end{aligned}  
\end{equation}
Similarly, $\bq$ can be computed via a backward recursion
\begin{equation}
    \label{eq:q by dynamic programming}
    q_N=\bhe'_N \cdot \bpsi'_N,\quad q_i=r'_i q_{i+1}+\bhe'_i \cdot \bpsi'_i, \quad i=N-1,\dots,1.
\end{equation}
Here $\br'$ is the ratio vector of $K_U$ and  $\mathcal{H}'=\{\bhe_i'\}_{i=1}^N$ is the set of block-edge vectors. $[\bpsi'_1,\dots,\bpsi'_N]$ is the corresponding partition of $\bpsi_{\zeta(1)+1:N}$, where each has the same dimension as $\bhe'_i$.

Using~\eqref{eq:p by dynamic programming} and~\eqref{eq:q by dynamic programming} to perform Sinkhorn iteration, we obtain FINOM for the one-dimensional case. The pseudo-code is summarized in Algorithm~\ref{alg:FINOM}.
\begin{algorithm}
\setstretch{1.15}
    \caption{1D FINOM}
    \label{alg:FINOM}
    \begin{algorithmic}[1]
        \Require{$\bx,\by \in \mbr^{N}$;\;$\;\bu,\bv \in \mbr^N; \;\;\rm{itr\underline{~}max}\in \mbn_+$; \;$\varepsilon,\rm{tol} \in \mbr_+$;\;$K \in \mbr^{ N\times N}$} 
        \State{Calculate the dividing index for $K,K^\top$; construct the vector representation $\br,\{\bhe_i\}$ for $K_L$,\;$\br',\{\bhe'_i\}$ for $K_U$,\;$\bs,\{\bxi_i\}$ for $K^\top_L$,\;$\bs',\{\bxi'_i\}$ for $K^\top_U$}
        \State {$\bphi,\bpsi\gets\frac{1}{N}\mathbf{1}_N$;\;$\bp,\bq \gets \mathbf{0}_N$;\;$\ell\gets0$}
        \While{($\ell<\rm{itr\underline{~}max}$) and ($\left\Vert\operatorname{diag}\left(\bpsi\right)K^\top\bphi-\bv\right\Vert_1>\rm{tol}$)}
        \State{$p_1\gets \bxi_1 \cdot\bphi_1,\quad q_N\gets \bxi'_N \cdot\bphi'_N$}
        \For{$i=1\;:\;N-1$}
        \State{$p_{i+1}\gets s_ip_i+\bxi_{i+1}\cdot\bphi_{i+1}$}
        \State{$q_{N-i}\gets s'_{N-i}q_{N-i+1}+\bxi'_{N-i}\cdot\bphi'_{N-i}$}
        \EndFor
        \State{$\bpsi\gets \bv \oslash \left(\bp+\bq\right)$}
        \State{$p_1\gets \bhe_1 \cdot\bpsi_1,\quad q_N\gets \bhe'_N \cdot\bpsi'_N$}
        \For{$i=1\;:\;N-1$}
        \State{$p_{i+1}\gets r_ip_i+\bhe_{i+1}\cdot\bpsi_{i+1}$}
        \State{$q_{N-i}\gets r'_{N-i}q_{N-i+1}+\bhe'_{N-i}\cdot\bpsi'_{N-i}$}
        \EndFor
        \State{$\bphi\gets \bu \oslash (\bp+\bq)$}
        \State{$\ell\gets\ell+1$}
        \EndWhile
        \State \vspace{2mm} 
        \Return {$\Gamma=\operatorname{diag}\left(\bphi\right)K\operatorname{diag}\left(\bpsi\right)$}
    \end{algorithmic}
\end{algorithm}

Note that the computation of $p_{i+1}$ in~\eqref{eq:p by dynamic programming} costs
$\left(\zeta(i+1)-\zeta(i)+1\right)$ multiplications and $\left(\zeta(i+1)-\zeta(i)\right)$ additions.
Iterating over all $i$, the total number of multiplications for computing $\bp$ is at most
$$
\zeta(1)+\sum\limits_{i=1}^{N-1}\left(\zeta(i+1)-\zeta(i)+1\right)=N+\zeta(N)-1,
$$
and the number of additions is at most 
$$
\zeta(1)-1+\sum\limits_{i=1}^{N-1}\left(\zeta(i+1)-\zeta(i)\right)=\zeta(N)-1.
$$
Analogously, computing $\bq$ by~\eqref{eq:q by dynamic programming} requires at most $\left(2N - \zeta(1) - 1\right)$ multiplications and $\left(N - \zeta(1) - 1\right)$ additions. Consequently, the matrix-vector multiplication  $K\bpsi$ requires no more than
\begin{equation}
\label{eq:number of multiplications}
    \left(N+\zeta(N)-1\right) + \left(2N - \zeta(1) - 1 \right) = 3N+\zeta(N)-\zeta(1)-2 < 4N,
\end{equation}
multiplications, and 
\begin{equation}
\label{eq:number of additions}
    \left(\zeta(N)-1\right) + \left(N - \zeta(1) - 1\right) + N = 2N+\zeta(N)-\zeta(1)-2 < 3N,
\end{equation}
additions. The additional term $+N$ in~\eqref{eq:number of additions} accounts for the extra additions incurred when computing the sum $\bp+\bq$. Therefore, the complexity of FINOM is $O(N)$.

\begin{remark}
\label{rem:log-domain-stabilization}
    In practice, the log-domain stabilization~\cite{chizat2018scaling} technique is widely used to ensure the numerical stability of the Sinkhorn algorithm, which can also be incorporated into our fast matrix-vector multiplication. The log-domain stabilization uses vectors $\ba$ and $\bb$ to ``absorb'' the excessive part of iterates $\bphi^{(t)}$ and  $\bpsi^{(t)}$:
    \begin{equation*}
        \ba\leftarrow\ba+\varepsilon\ln\brac{\bphi^{(t)}},\quad 
        \bb\leftarrow\bb+\varepsilon\ln\brac{\bpsi^{(t)}},\quad
        \bphi^{(t)}\leftarrow\mathbf{1}_N,\quad
        \bpsi^{(t)}\leftarrow\mathbf{1}_N,
    \end{equation*}
    and the kernel is replaced with $\operatorname{diag}\left(e^{\ba/\varepsilon}\right)K\operatorname{diag}\left(e^{\bb/\varepsilon}\right)$. Correspondingly, we have to replace $\br$ and $\bhe$ in~\eqref{eq:ratio} and~\eqref{eq:eta} with
    \begin{equation*}
        r_i\leftarrow e^{\brac{a_{i+1}-a_i}/\varepsilon}\cdot r_i,\quad \check{k}_{i,j}\leftarrow e^{\brac{a_{i}+b_j}/\varepsilon}\check{k}_{i,j}.
    \end{equation*}
    The updates for $K_U$ and for the transposed case are analogous. This leads to the stabilized fast algorithm with linear complexity.
\end{remark}

\section{Extension to high dimensions}
\label{sec:2d}
In this section, we discuss the extension of FINOM to the 2D case, which already contains the essential idea of the higher-dimensional extension. More precisely, the 2D and higher-dimensional algorithms are obtained by applying the 1D fast matrix-vector multiplication in Section~\ref{subsec:fmvm} successively along each coordinate direction.

To avoid confusion, we record the following notation.

\begin{itemize}[leftmargin=2.2em,itemsep=1pt,topsep=2pt]

    \item $K\in \mbr^{NM\times NM}$: the 2D kernel matrix, which will be shown to satisfy $K=K_y\otimes K_x$;
    
    \item $K_{[i,j]}$: the $(i,j)$-th block of $K$;
    
    \item $K_x\in\mbr^{N\times N}$: the kernel matrix generated by the 1D point sets $\{x_k^1\}_{k=1}^N$ and $\{x_l^2\}_{l=1}^N$ along the $x$-direction;
    
    \item $K_y\in\mbr^{M\times M}$: the kernel matrix generated by the 1D point sets $\{y_i^1\}_{i=1}^M$ and $\{y_j^2\}_{j=1}^M$ along the $y$-direction;
    
    \item For any matrix, its $(i,j)$-th entry is denoted by $(\cdot)_{ij}$ (e.g., $(K_x)_{ij}$, $(K_y)_{ij}$, or $(K_{[i,j]})_{i'j'}$);
    
    \item $\operatorname{vec}(\cdot)$: the column-major vectorization of a matrix.
\end{itemize}

\subsection{The 2D kernel matrix}
Consider two probability distributions
$$
\bu=
\begin{pmatrix}
    u_{11} & u_{12} & \cdots & u_{1M}\\
    u_{21} & u_{22} & \cdots & u_{2M}\\
    \vdots & \vdots & \ddots & \vdots\\
    u_{N1} & u_{N2} & \cdots & u_{NM}
\end{pmatrix},\qquad
\bv=
\begin{pmatrix}
    v_{11} & v_{12} & \cdots & v_{1M}\\
    v_{21} & v_{22} & \cdots & v_{2M}\\
    \vdots & \vdots & \ddots & \vdots\\
    v_{N1} & v_{N2} & \cdots & v_{NM}
\end{pmatrix},
$$
on two 2D non-uniform meshes
$$
\bx^1\times \by^1,\qquad \bx^2\times \by^2,
$$
where $\bx^1,\bx^2\in\mbr^N$, $\by^1,\by^2\in\mbr^M$, and $\times$ represents the Cartesian product of the 1D positions. We order the grid points in column-major order. The corresponding kernel matrix $K\in\mbr^{NM\times NM}$ can then be written as an $M\times M$ block matrix with blocks of size $N\times N$:
$$
\renewcommand*{\arraystretch}{2}
 \setlength{\arraycolsep}{3pt} 
K=\left(
\begin{array}{c|c|c|c|c}
	K_{[1,1]}&K_{[1,2]}&K_{[1,3]}&\cdots &K_{[1,M]}\\
	\hline
	K_{[2,1]}&K_{[2,2]}&K_{[2,3]}&\cdots&K_{[2,M]}\\
	\hline
	K_{[3,1]}&K_{[3,2]}&K_{[3,3]}&\cdots&K_{[3,M]}\\
	\hline
	\vdots&\vdots&\vdots&\ddots&\vdots \\
	\hline
	K_{[M,1]}&K_{[M,2]}&K_{[M,3]}&\cdots&K_{[M,M]}\\
\end{array}
\right).
$$
For each fixed $1\le i,j\le M$, the block $K_{[i,j]}$ is the $N\times N$ kernel matrix generated by the two sets of points
$$
\{(x^1_k,y^1_i)\}_{k=1}^{N},\qquad \{(x^2_l,y^2_j)\}_{l=1}^{N}.
$$
In other words, within the block $K_{[i,j]}$, the second coordinates $y^1_i$ and $y^2_j$ are fixed, while the first coordinates $x^1_k$ and $x^2_l$ vary. Therefore,
\begin{equation}
\label{eq:2d-block-entry}
\left(K_{[i,j]}\right)_{kl}
=e^{-\brac{\abs{x^1_k-x^2_l}+\abs{y^1_i-y^2_j}}/\veps}
=e^{-\abs{x^1_k-x^2_l}/\veps}\,e^{-\abs{y^1_i-y^2_j}/\veps}.
\end{equation}
This shows that every block shares the same matrix and differs only by a scalar factor $e^{-\abs{y^1_i-y^2_j}/\veps}$. More precisely, let
\begin{equation}
\label{eq:2d-1d-kernels}
K_x=\left[e^{-\abs{x^1_k-x^2_l}/\veps}\right]_{k,l=1}^{N}\in\mbr^{N\times N},\qquad
K_y=\left[e^{-\abs{y^1_i-y^2_j}/\veps}\right]_{i,j=1}^{M}\in\mbr^{M\times M},
\end{equation}
be the 1D kernel matrices in the two coordinate directions, then
\begin{equation*}
K_{[i,j]}=(K_y)_{ij}K_x,\qquad 1\le i,j\le M.
\end{equation*}
Consequently,
\begin{equation}
\label{eq:2d-kron}
K=K_y\otimes K_x,
\end{equation}
where $\otimes$ is the Kronecker product.

\subsection{Fast matrix-vector multiplication}
The structure \eqref{eq:2d-kron}, together with the 1D fast matrix-vector multiplication developed in Section~\ref{subsec:fmvm}, enables the efficient computation of $K\bpsi$ and $K^\top\bphi$ for $\bphi,\bpsi\in\mbr^{NM}$.

Reshape $\bpsi$ into an $N\times M$ matrix $\Psi$ in column-major order, so that $\bpsi=\operatorname{vec}(\Psi)$. Using~\eqref{eq:2d-kron} and the identity 
\[\operatorname{vec}(AXB)=\left(B^\top\otimes A\right)\operatorname{vec}(X),\]
where $A,X,$ and $B$ are matrices such that the product $AXB$ is well-defined, we have
\begin{equation}
\label{eq:2d-kpsi}
K\bpsi=(K_y\otimes K_x)\bpsi=\operatorname{vec}\brac{(K_x\Psi)K_y^\top}.
\end{equation}
In the identity above, by taking $A=K_x,\;X=\Psi,$ and $B=K_y^\top$, we directly obtain the expression on the right-hand side.
Similarly, reshaping $\bphi$ into an $N\times M$ matrix $\Phi$ in column-major order, so that $\bphi=\operatorname{vec}(\Phi)$, we obtain
\begin{equation}
\label{eq:2d-ktphi}
K^\top\bphi=(K_y^\top\otimes K_x^\top)\bphi=\operatorname{vec}\brac{(K_x^\top\Phi)K_y}.
\end{equation}
Incorporating \eqref{eq:2d-kpsi} and \eqref{eq:2d-ktphi} into the Sinkhorn iteration produces the  2D FINOM, as summarized in Algorithm~\ref{alg:2D-FINOM}.

\begin{algorithm}[H]
\setstretch{1.15}
    \caption{2D FINOM}
    \label{alg:2D-FINOM}
    \begin{algorithmic}[1]
        \Require{$\bx^1,\bx^2\in\mbr^{N}$; $\by^1,\by^2\in\mbr^{M}$; $\bu,\bv\in\mbr^{N\times M}$; $\rm{itr\underline{~}max}\in\mbn_+$; $\veps,\rm{tol}\in\mbr_+$}
        \State{$\Phi,\Psi \leftarrow \frac{1}{NM}\mathbf{1}_{N\times M}$; $\bphi\leftarrow\operatorname{vec}(\Phi)$; $\bpsi\leftarrow\operatorname{vec}(\Psi)$}
        \State{$S,T\in\mbr^{N\times M}$; $\ell\leftarrow 0$}
        \State{Construct the vector representations of $K_x$, $K_x^\top$, $K_y$, and $K_y^\top$ with $\bx^1,\bx^2,\by^1,\by^2$}
        \While{$(\ell<\rm{itr\underline{~}max})$ and $\brac{\norm{\operatorname{diag}(\bpsi)K^\top\bphi-\operatorname{vec}(\bv)}_1>\rm{tol}}$}
        \State{$S\leftarrow K_x^\top\Phi$, $T\leftarrow SK_y$}
        \State{$\Psi\leftarrow \bv\oslash T$, $\bpsi\leftarrow\operatorname{vec}(\Psi)$}
        \State{$S\leftarrow K_x\Psi$, $T\leftarrow SK_y^\top$}
        \State{$\Phi\leftarrow \bu\oslash T$, $\bphi\leftarrow\operatorname{vec}(\Phi)$}
        \State{$\ell\leftarrow\ell+1$}
        \EndWhile
        \State \vspace{2mm}
        \Return{$\Gamma=\operatorname{diag}(\bphi)K\operatorname{diag}(\bpsi)$}
    \end{algorithmic}
\end{algorithm}
We now analyze the computational complexity of Algorithm~\ref{alg:2D-FINOM}. The product $S\leftarrow K_x\Psi$ in line 7 can be computed by applying the 1D fast matrix-vector multiplication to each of the $M$ columns of $\Psi$, yielding a total cost of $O(MN)$. Then $T\leftarrow SK_y^\top$ is carried out by applying the same 1D fast matrix-vector multiplication to each of the $N$ rows of $S$, with a total cost of $O(NM)$. Hence, $K\bpsi$ in \eqref{eq:2d-kpsi} can be evaluated in $O(NM)$ operations. In line 5, by a similar columnwise/rowwise strategy,  $K^\top\bphi$ in \eqref{eq:2d-ktphi} can also be computed in $O(NM)$ operations. Consequently, each iteration of Algorithm~\ref{alg:2D-FINOM} has linear complexity with respect to the total number of grid points $NM$.

\vspace{0.4cm} 

\begin{remark}
\label{rem:2d-log-stab}
The 2D FINOM can also be combined with log-domain stabilization. Although the absorbed kernel
\[
K'=\operatorname{diag}(e^{\ba/\veps})K\operatorname{diag}(e^{\bb/\veps})
\]
generally no longer admits a direct Kronecker product representation, the matrix-vector multiplications
$K'\bpsi$ and $(K')^\top\bphi$ can still be carried out efficiently by applying the 1D stabilized fast algorithm
successively along each coordinate direction. In this way, the stabilized 2D algorithm preserves the same
linear complexity $O(NM)$ with respect to the total number of grid points. For completeness, the detailed
derivation and the corresponding algorithm are provided in Appendix~\ref{app:2d-log-stab}.
\end{remark}

\vspace{0.4cm} 

\begin{remark}
The same idea extends directly to higher dimensions. If the mesh has sizes $n_1,\dots,n_d$ and the kernel factorizes as $K=K_d\otimes\cdots\otimes K_1$, then each application of $K$ or $K^\top$ can be carried out successively along each coordinate direction by the same 1D fast algorithm. Therefore, the per-iteration complexity remains linear in the total number of grid points $\prod_{\alpha=1}^{d}n_\alpha$.
\end{remark}

\section{Numerical Experiments}
\label{sec:experiments}
In this section, we compare FINOM with the standard Sinkhorn algorithm on several one- and two-dimensional test problems. Since FINOM accelerates the matrix-vector multiplications in Sinkhorn without changing the underlying entropic iteration, the two methods are compared under the same regularization parameter, initialization, iteration number, and implementation environment. Both algorithms are implemented in C++. All experiments are conducted on a platform with 128GB RAM and one 14-core Intel(R) Xeon(R) Gold 5117 CPU @2.00GHz. Unless otherwise stated, log-domain stabilization is enabled in both methods for numerical robustness.

For each test, we report the averaged computational time, the corresponding speed-up ratio, and the Frobenius norm of the difference between the transport plans produced by the two algorithms. For brevity, this difference is denoted by $\Delta \Gamma$ in the tables. To compare convergence in wall-clock time, we also plot the time required to reach prescribed marginal errors. Here the marginal error at iteration $\ell$ is defined by
$\norm{\operatorname{diag}\brac{\bpsi^{(\ell)}}K^\top \bphi^{(\ell)}-\bv}_1.$
Unless otherwise stated, all timing results reported in the tables and in the left panels of the timing figures are averaged over 100 independent trials. In contrast, the right panels record the wall-clock time required to reach prescribed marginal errors.

\subsection{1D random distributions on Chebyshev nodes}
We first consider the Wasserstein-1 distance between two random distributions on $[0,1]$, both supported on the $N$ Chebyshev nodes
\begin{equation*}
    x_k=y_k=\frac{1}{2}+\frac{1}{2}\cos\brac{\frac{2k-1}{2N}\pi}, \qquad k=1,\dots,N.
\end{equation*}
The entries of the mass vectors $\bu$ and $\bv$ are independently sampled from the uniform distribution on $[0,1]$ and then normalized to sum to one. Each trial runs 1,000 Sinkhorn iterations with regularization parameter $\varepsilon=0.01$.

\begin{table}[htbp]
    \centering
    \begin{tabular}{ccccc}
        \toprule
        \multirow{2}{*}{$N$} & \multicolumn{2}{c}{Computational time (s)} & \multirow{2}{*}{Speed-up ratio} & \multirow{2}{*}{$\norm{\Delta\Gamma}_F$}\\
        \cmidrule(lr){2-3}
        & FINOM & Sinkhorn & & \\
        \midrule
        500  & $7.05\times 10^{-3}$ & $7.77\times 10^{-1}$ & $1.11\times 10^{2}$ & $4.94\times 10^{-17}$\\
        2000 & $2.60\times 10^{-2}$ & $3.05\times 10^{1}$  & $1.17\times 10^{3}$ & $9.26\times 10^{-18}$\\
        8000 & $1.41\times 10^{-1}$ & $9.05\times 10^{2}$  & $6.46\times 10^{3}$ & $5.36\times 10^{-18}$\\
        \bottomrule
    \end{tabular}
    \caption{1D random distributions on Chebyshev nodes. Comparison between FINOM and the standard Sinkhorn algorithm for different values of $N$ with $\varepsilon=0.01$. Columns 2--4 report the averaged computational times and the corresponding speed-up ratios. Column 5 reports the Frobenius norm of the difference between the transport plans returned by the two methods.}
    \label{table:1Dchebyshev}
\end{table}

\begin{figure}[htbp]
    \centering
    \begin{minipage}{0.48\textwidth}
        \centering
        \includegraphics[trim=3.4cm 9cm 3.4cm 9cm,clip,width=\textwidth]{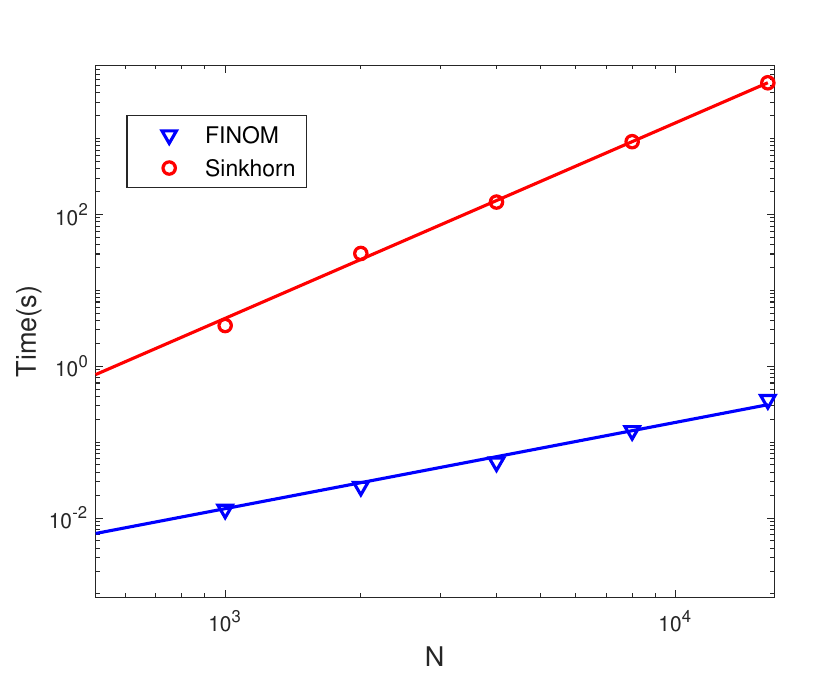}
    \end{minipage}
    \hfill
    \begin{minipage}{0.48\textwidth}
        \centering
        \includegraphics[trim=3.4cm 9cm 3.4cm 9cm,clip,width=\textwidth]{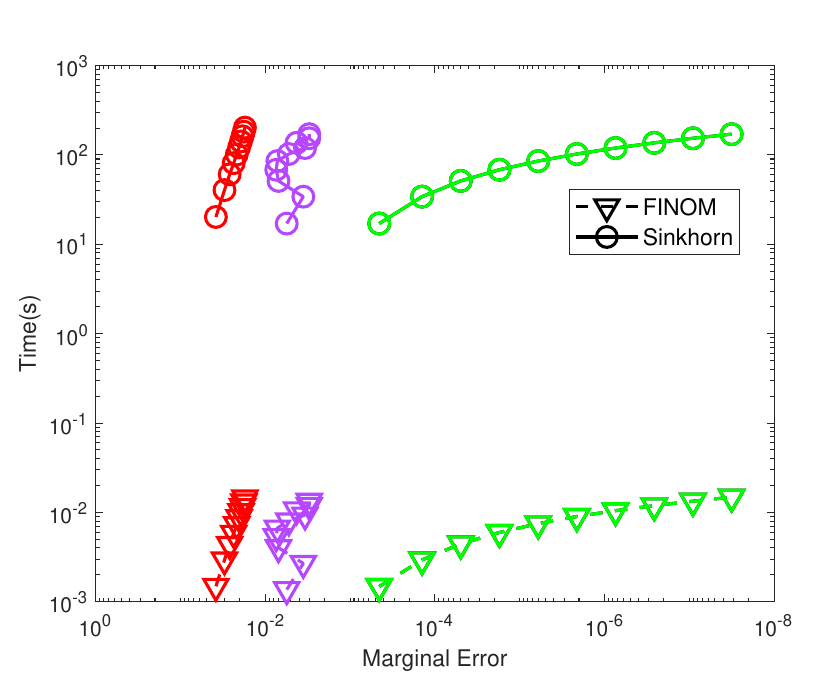}
    \end{minipage}
    \caption{1D random distributions on Chebyshev nodes. Left: computational time versus $N$ with $\varepsilon=0.01$. Right: computational time required to reach prescribed marginal errors for $\varepsilon=0.1$ (green), $0.01$ (pink), and $0.001$ (red), with $N=10,000$.}
    \label{fig:1Dchebyshev}
\end{figure}

Table~\ref{table:1Dchebyshev} shows that FINOM achieves substantial acceleration while producing transport plans that are essentially indistinguishable from those returned by the standard Sinkhorn algorithm. Over the tested range, a least-squares fit of the timing curves in Figure~\ref{fig:1Dchebyshev} (left) gives observed scalings of approximately $O\brac{N^{1.13}}$ for FINOM and $O\brac{N^{2.58}}$ for Sinkhorn. Figure~\ref{fig:1Dchebyshev} (right) further shows that, for the fixed problem size $N=10,000$, FINOM reaches the same marginal-error levels several orders of magnitude faster for all three choices of $\varepsilon$.

\subsection{1D random distributions on random nodes}
We next consider a second one-dimensional random test. The source and target nodes are generated independently by sampling from the uniform distribution on $[0,1]$ and then sorting in ascending order. The entries of the two mass vectors are again independently sampled from the uniform distribution on $[0,1]$ and normalized to sum to one. The regularization parameter and the number of iterations are kept the same as in the previous subsection. The numerical results are summarized in Table~\ref{table:1Duniform} and Figure~\ref{fig:1Duniform}.

\begin{table}[htbp]
    \centering
    \begin{tabular}{ccccc}
        \toprule
        \multirow{2}{*}{$N$} & \multicolumn{2}{c}{Computational time (s)} & \multirow{2}{*}{Speed-up ratio} & \multirow{2}{*}{$\norm{\Delta\Gamma}_F$}\\
        \cmidrule(lr){2-3}
        & FINOM & Sinkhorn & & \\
        \midrule
        500  & $6.48\times 10^{-3}$ & $8.23\times 10^{-1}$ & $1.27\times 10^{2}$ & $3.49\times 10^{-17}$\\
        2000 & $5.87\times 10^{-2}$ & $2.76\times 10^{1}$  & $4.68\times 10^{2}$ & $1.35\times 10^{-17}$\\
        8000 & $3.90\times 10^{-1}$ & $8.87\times 10^{2}$  & $2.27\times 10^{3}$ & $1.30\times 10^{-17}$\\
        \bottomrule
    \end{tabular}
    \caption{1D random distributions on random nodes. Comparison between FINOM and the standard Sinkhorn algorithm for different values of $N$ with $\varepsilon=0.01$. Columns 2--4 report the averaged computational times and the corresponding speed-up ratios. Column 5 reports the Frobenius norm of the difference between the transport plans returned by the two methods.}
    \label{table:1Duniform}
\end{table}

\begin{figure}[htbp]
    \centering
    \begin{minipage}{0.48\textwidth}
        \centering
        \includegraphics[trim=3.4cm 9cm 3.4cm 9cm,clip,width=\textwidth]{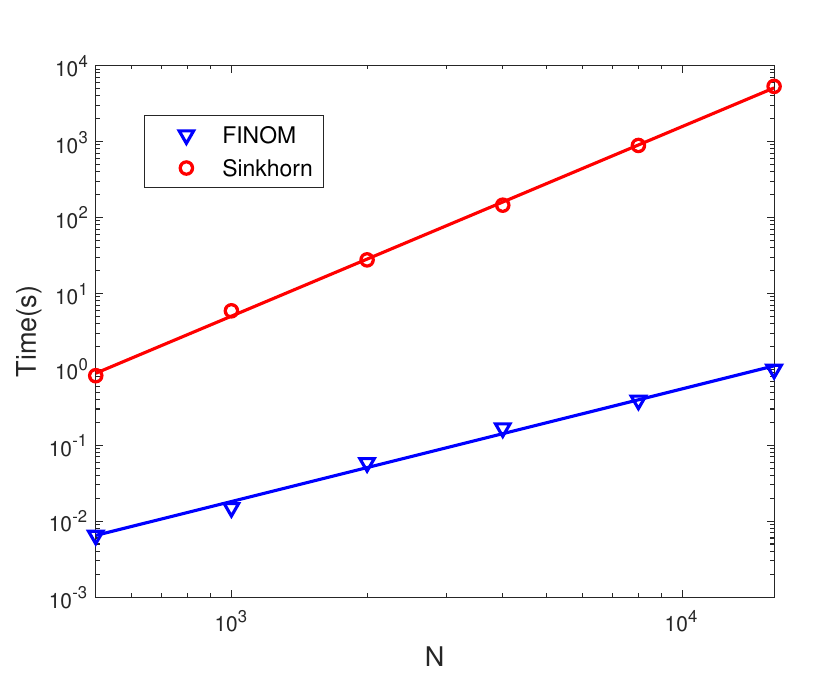}
    \end{minipage}
    \hfill
    \begin{minipage}{0.48\textwidth}
        \centering
        \includegraphics[trim=3.4cm 9cm 3.4cm 9cm,clip,width=\textwidth]{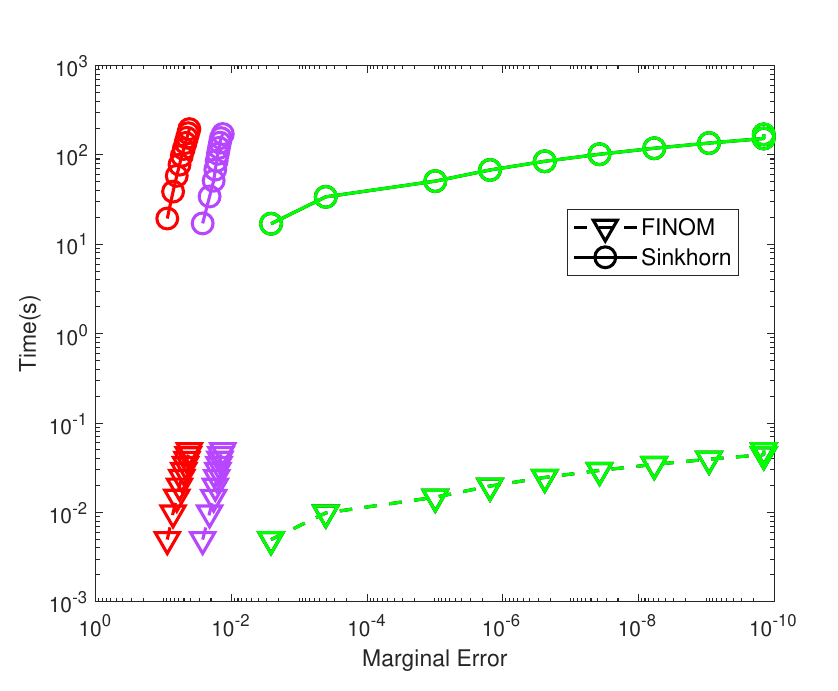}
    \end{minipage}
    \caption{1D random distributions on random nodes. Left: computational time versus $N$ with $\varepsilon=0.01$. Right: computational time required to reach prescribed marginal errors for $\varepsilon=0.1$ (green), $0.01$ (pink), and $0.001$ (red), with $N=10,000$.}
    \label{fig:1Duniform}
\end{figure}

The same conclusion is observed in this second 1D experiment. Table~\ref{table:1Duniform} shows that FINOM remains much faster than the standard Sinkhorn algorithm while the difference between the resulting transport plans stays at the level of machine precision. Compared with the Chebyshev-node case, the observed speed-ups remain of the same order, indicating that the acceleration is robust with respect to irregular node locations. A least-squares fit of the timing curves in Figure~\ref{fig:1Duniform} (left) gives observed scalings of approximately $O\brac{N^{1.18}}$ for FINOM and $O\brac{N^{2.50}}$ for Sinkhorn. Figure~\ref{fig:1Duniform} (right) shows that the wall-clock advantage of FINOM persists across a range of target marginal errors and regularization parameters.

\subsection{2D random distributions}
We finally test the two-dimensional version of FINOM. The two probability arrays are supported on Cartesian-product meshes of size $N\times N$. More precisely, the two meshes are given by $\bx^1\times\by^1$ and $\bx^2\times\by^2$, where the four one-dimensional node vectors $\bx^1,\bx^2,\by^1,\by^2\in\mbr^N$ are generated independently by sampling from the uniform distribution on $[0,1]$ and then sorting in ascending order. The entries of the two probability arrays are also independently sampled from the uniform distribution on $[0,1]$ and then normalized. Each trial runs 1,000 iterations with regularization parameter $\varepsilon=0.01$. The averaged computational times are reported in Table~\ref{table:2Duniform}, while Figure~\ref{fig:2Duniform} displays the corresponding timing curves and the time-to-error comparisons.

\begin{table}[htbp]
    \centering
    \begin{tabular}{ccccc}
        \toprule
        \multirow{2}{*}{$N\times N$} & \multicolumn{2}{c}{Computational time (s)} & \multirow{2}{*}{Speed-up ratio} & \multirow{2}{*}{$\norm{\Delta\Gamma}_F$}\\
        \cmidrule(lr){2-3}
        & FINOM & Sinkhorn & & \\
        \midrule
        $10\times 10$   & $8.10\times 10^{-3}$ & $4.33\times 10^{-2}$ & $5.35\times 10^{0}$ & $2.27\times 10^{-17}$\\
        $20\times 20$   & $1.80\times 10^{-2}$ & $5.26\times 10^{-1}$ & $2.92\times 10^{1}$ & $9.74\times 10^{-17}$\\
        $40\times 40$   & $5.70\times 10^{-2}$ & $1.49\times 10^{1}$  & $2.62\times 10^{2}$ & $4.33\times 10^{-17}$\\
        $80\times 80$   & $2.19\times 10^{-1}$ & $4.41\times 10^{2}$  & $2.01\times 10^{3}$ & $1.16\times 10^{-17}$\\
        $160\times 160$ & $9.35\times 10^{-1}$ & $1.33\times 10^{4}$  & $1.43\times 10^{4}$ & $8.91\times 10^{-18}$\\
        \bottomrule
    \end{tabular}
    \caption{2D random distributions. Comparison between FINOM and the standard Sinkhorn algorithm for different grid sizes $N\times N$ with $\varepsilon=0.01$. Columns 2--4 report the averaged computational times and the corresponding speed-up ratios. Column 5 reports the Frobenius norm of the difference between the transport plans returned by the two methods.}
    \label{table:2Duniform}
\end{table}

\begin{figure}[htbp]
    \centering
    \begin{minipage}{0.48\textwidth}
        \centering
        \includegraphics[trim=3.4cm 9cm 3.4cm 9cm,clip,width=\textwidth]{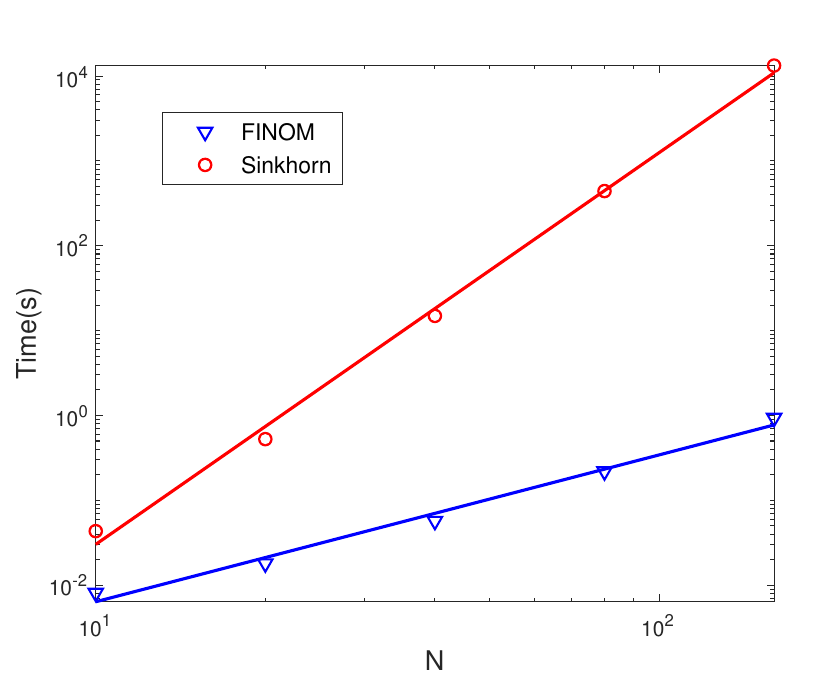}
    \end{minipage}
    \hfill
    \begin{minipage}{0.48\textwidth}
        \centering
        \includegraphics[trim=3.4cm 9cm 3.4cm 9cm,clip,width=\textwidth]{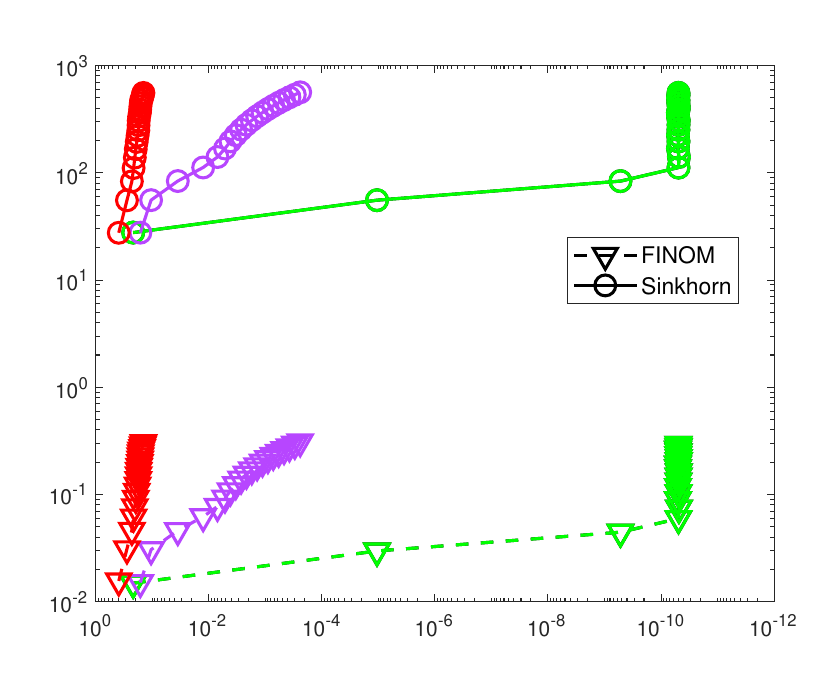}
    \end{minipage}
    \caption{2D random distributions. Left: computational time versus the grid size $N\times N$ with $\varepsilon=0.01$. Right: computational time required to reach prescribed marginal errors for $\varepsilon=0.1$ (green), $0.01$ (pink), and $0.001$ (red), with $N\times N=100\times 100$.}
    \label{fig:2Duniform}
\end{figure}

The results in Table~\ref{table:2Duniform} confirm that the advantage of FINOM is even more pronounced in two dimensions. The transport plans computed by FINOM and Sinkhorn remain essentially identical, while the speed-up ratio grows rapidly with the problem size and exceeds $10^4$ for the largest test reported here. This indicates that the proposed acceleration remains effective on genuinely random non-uniform meshes in two dimensions. Since the total number of grid points is $N^2$, a least-squares fit of the timing curves in Figure~\ref{fig:2Duniform} (left) gives observed scalings of approximately $O\brac{N^{1.74}}$ for FINOM and $O\brac{N^{4.62}}$ for Sinkhorn. We do not interpret the exponent $1.74$ literally as a complexity better than linear, but it indicates that, on the present range of problem sizes, FINOM exhibits the expected near-linear growth with respect to the total number of grid points $N^2$, in sharp contrast with the much steeper growth of the standard Sinkhorn algorithm. Figure~\ref{fig:2Duniform} (right) shows that this wall-clock advantage persists when the target marginal error is varied.

\section{Conclusion}
\label{sec:conclusion}
In this paper, we proposed FINOM, a fast Sinkhorn algorithm for computing the entropy-regularized Wasserstein-1 distance on non-uniform meshes. In one dimension, we introduced the dividing index and showed that the resulting partitioned kernel matrices possess a quasi-collinear structure, which leads to a linear-complexity dynamic programming method for matrix-vector multiplication.  In two and higher dimensions, the fast algorithm is obtained by applying the one-dimensional fast multiplication successively along each coordinate direction, so that the per-iteration complexity remains linear in the total number of grid points.
Numerical experiments on 1D and 2D random distributions show that FINOM significantly outperforms the standard Sinkhorn algorithm in computational time while producing essentially identical transport plans. The observed timings are consistent with the theoretical $O(N)$ behavior in 1D and $O(NM)$ behavior in 2D. 

\section*{Acknowledgements}
The work is supported by the National Natural Science Foundation of China (Grant No. 12271289).

\newpage
\bibliographystyle{siam}
\bibliography{ref.bib}
\newpage
\appendix
\section{2D FINOM with log-domain stabilization}
\label{app:2d-log-stab}

In this appendix, we provide the detailed derivation mentioned in Remark~\ref{rem:2d-log-stab}. While the Kronecker product representation enables efficient computation, it makes log-domain stabilization~\cite{chizat2018scaling} harder to exploit. Under log-domain stabilization, the kernel becomes $K'=\operatorname{diag}(e^{\ba/\varepsilon})K\operatorname{diag}(e^{\bb/\varepsilon})$, with arbitrary absorption vectors $\ba,\bb\in\mbr^{NM}$. However, $K'$ generally lacks a Kronecker factorization compatible with $K = K_y \otimes K_x$, because the absorption vectors $\ba$ and $\bb$ are typically not separable across dimensions. This prevents direct Kronecker-based fast computation under log-domain stabilization.

To address this issue, we do not work directly with the matrix $K'$. Instead, we evaluate $K'\bpsi$ and $(K')^\top\bphi$ through decoupled 1D steps, so that the Kronecker structure of $K$ can still be exploited after incorporating the absorption vectors. 
This decoupled procedure extends the one-dimensional stabilization in Remark~\ref{rem:log-domain-stabilization} to higher dimensions without increasing complexity, thereby preserving numerical stability and computational efficiency.

We illustrate the idea using $K'\bpsi$; the case of $(K')^\top \bphi$ is handled in the same way. 
Specifically, reshape $\bb$ into $B \in \mbr^{N \times M}$ in column-major order. Then, 
\begin{equation}
    \label{eq:K'bpsi}
    \begin{aligned}
            K'\bpsi=\operatorname{diag}(e^{\ba/\varepsilon})&K\operatorname{diag}(e^{\bb/\varepsilon})\bpsi \\[0.6em]
            = e^{\ba/\veps}& \odot \left(K \left(e^{\bb/\veps} \odot \bpsi\right)\right) = e^{\ba/\veps}\odot \left(\left(K_y\otimes K_x\right) \left(e^{\bb/\veps} \odot \bpsi\right)\right)\\
            & \qquad\qquad\qquad\qquad\qquad\qquad= e^{\ba/\veps}\odot\mathrm{vec}\left(\left(K_x\left(e^{B/\veps}\odot\Psi\right)\right)K_y^\top\right),
    \end{aligned}
\end{equation}
where $\bpsi=\operatorname{vec}(\Psi)$ and $\odot$ denotes element-wise multiplication. 

Let $Q=K_x\left(e^{B/\veps}\odot\Psi\right) \in \mbr^{N \times M}$. For each $j=1,\dots,M$, the $j$-th column
\begin{equation*}
    Q(:,j) = K_x \left(e^{B(:,j)/\veps} \odot \Psi(:,j)\right)=K_x\operatorname{diag}\left(e^{B(:,j)/\veps} \right)\Psi(:,j)\;,
\end{equation*}
can be computed using the 1D stabilized fast matrix-vector multiplication $K_x\Psi(:,j)$, with absorption vectors $\ba_x = \mathbf{0}_N$, $\bb_x = B(:,j)$ in Remark~\ref{rem:log-domain-stabilization}. Together, these $M$ columns form the full matrix $Q$. Each computation takes $O(N)$, resulting in a total complexity of $O(NM)$. 

Reshaping $\ba$ into $A \in \mbr^{N \times M}$ in column-major order, \eqref{eq:K'bpsi} can be rewritten as
\begin{equation}
K'\bpsi=e^{\ba/\veps}\odot\mathrm{vec}\left(QK_y^\top\right)=\mathrm{vec}\left(e^{A/\veps}\odot QK_y^\top\right).
\end{equation}
Let $P=e^{A/\veps}\odot QK_y^\top \in \mbr^{N \times M}$. Similar to the computation of $Q$, the matrix $P$ can be computed row by row. For each $k=1,\dots,N$, the $k$-th row satisfies
\begin{equation*}
    P(k,:) = e^{A(k,:)/\varepsilon} \odot\left( Q(k,:)K_y^\top \right)=\left(\operatorname{diag}\left(e^{A(k,:)^\top/\varepsilon}\right)K_yQ(k,:)^\top\right)^\top, 
\end{equation*}
Therefore, $P(k,:)$ can be obtained by applying the 1D stabilized fast algorithm to $K_yQ(k,:)^\top$ with left absorption $\ba_y=A(k,:)^\top$ and right absorption $\bb_y=\mathbf{0}_M$. Each of the $N$ row computations costs $O(M)$, leading to a total complexity of $O(NM)$.

Finally, $K'\bpsi=\operatorname{vec}(P)$, so the overall complexity remains $O(NM)$. The computation of $(K')^\top\bphi$ is handled in the same way. Therefore, by replacing the four matrix products in Algorithm~\ref{alg:2D-FINOM} with their stabilized counterparts, we obtain the stabilized 2D FINOM algorithm without changing the per-iteration complexity.
\end{sloppypar}
\end{document}